\documentclass[12pt]{amsart}
\usepackage{amsfonts}
\usepackage{amssymb,mathrsfs}
\usepackage{amsmath}
\usepackage{enumerate}
\usepackage{latexsym}
\usepackage{graphicx}
\usepackage{bm}
\usepackage{subfigure}
\usepackage{float}
\usepackage{amsmath}
\usepackage{amsthm}
\usepackage{verbatim}
\usepackage{vmargin}
\usepackage{amstext}
\usepackage{array}
\usepackage{booktabs}
\usepackage{indentfirst}
\usepackage[compress]{cite}
\usepackage{url}
\usepackage{geometry}
\numberwithin{equation}{section}

\newtheorem{thm}{Theorem}[section]
\newtheorem{cor}[thm]{Corollary}
\newtheorem{lma}[thm]{Lemma}

\newtheorem{defn}[thm]{Definition}
\newtheorem{rem}[thm]{Remark}

\theoremstyle{definition}

\begin{document}
\setlength{\oddsidemargin}{80pt}
\setlength{\evensidemargin}{80pt}
\setlength{\topmargin}{80pt}
\author{Jun~Xian, Xiaoda~Xu}

\address{J.~Xian\\Department of Mathematics and Guangdong Province Key Laboratory of Computational Science
 \\Sun
Yat-sen University
 \\
510275 Guangzhou\\
China.} \email{xianjun@mail.sysu.edu.cn}

\address{X.~Xu\\Department of Mathematics
 \\Sun
Yat-sen University
 \\
510275 Guangzhou\\
China.} \email{xuxd26@mail2.sysu.edu.cn}

\title[Bounds of star discrepancy for HSFC]{STAR DISCREPANCY BOUNDS BASED ON HILBERT SPACE FILLING CURVE STRATIFIED SAMPLING AND ITS APPLICATIONS}


\date{\today}

\begin{abstract}
In this paper, we consider the upper bound of the probabilistic star discrepancy based on Hilbert space filling curve sampling. This problem originates from the multivariate integral approximation, but the main result removes the strict conditions on the sampling number of the classical grid-based jittered sampling. The main content has three parts. First, we inherit the advantages of this new sampling and achieve a better upper bound of the random star discrepancy than the use of Monte Carlo sampling. In addition, the convergence order of the upper bound is improved from $O(N^{-\frac{1}{2}})$ to $O(N^{-\frac{1}{2}-\frac{1}{2d}})$. Second, a better uniform integral approximation error bound of the function in the weighted space is obtained. Third, other applications will be given. Such as the sampling theorem in Hilbert spaces and the improvement of the classical Koksma-Hlawka inequality. Finally, the idea can also be applied to the proof of the strong partition principle of the star discrepancy version.
\end{abstract}

\maketitle

\section{Introduction}

Many problems in industry, statistics, physics and finance need to introduce multivariable integrals. In order to effectively represent, analyze and process these integrals on the computer, it is necessary to discretize the sample points. A natural problem is whether these discrete sampling values can achieve the optimal fast convergence approximation of multivariable integrals. This is investigated by the discrepancy function method and sampling theory. And which also makes it a popular research topic in applied mathematics and engineering science.

For the multivariate function integral $\int_{[0,1]^{d}} f(x) dx$ defined on a $d-$dimensional unit cube, if a discrete sampling point set $P_{N,d}=\{t_{1},t_{2},\ldots,t_{N}\}$ is selected, it can be approximated by the sample mean $\frac{1}{N} \sum_{t \in P_{N, d}} f(t)$. Then the approximation error can be calculated by:

\begin{equation*}
  E_N(f)=\left|\int_{[0,1]^{d}} f(x) d x-\frac{1}{N} \sum_{t \in P_{N, d}} f(t)\right|.
\end{equation*}

Among various techniques to estimate the error of $E_N(f)$, star discrepancy has proven to be among the most efficient approaches. Star discrepancy was first used to measure the irregularity of the point distribution. For a point set $P_{N,d}=\{t_{1},t_{2},\ldots,t_{N}\}$, the star discrepancy is defined as follows:

\begin{equation}\label{*-d}D_{N}^{*}\left(P_{N,d}\right):=\sup _{x \in[0,1]^{d}}\left|\lambda([0, x])-\frac{\sum_{n=1}^{N} I_{[0, x]}\left(t_{n}\right)}{N}\right|,\end{equation}
where $\lambda$ denotes the $d$-dimensional Lebesgue measure, and $I_{[0,x]}$ denotes the characteristic function defined on the $d$-dimensional rectangle $[0,x]$.

A well-known result in discrepancy theory is the Koksma-Hlawka inequality~\cite{hlawka1961funktionen, koksma1942een}. It shows that the integral approximation error of any multivariate function can be separated into two parts, one is the star discrepancy of the defined point set $P_{N,d}$, and the other is the total variation of the function $f$ (under the definition of Hardy and Krause), i.e,:

\begin{equation}\label{K-H}
E_N(f)\leq D_{N}^{*}\left(P_{N, d}\right) V(f).\end{equation}

The equation \eqref{K-H} indicates that when $f$ is a function of bounded variation, a smaller upper bound on the star discrepancy yields a more accurate approximation error of the integration. Therefore, we focus on the research of the star discrepancy in the following.

There are many kinds of constructive sampling sets, also known as \textbf{low discrepancy point sets}~\cite{2010Digital}, which obtain a satisfactory convergence rate of star discrepancy bound. For a fixed dimension $d$, it can reach $O((\ln N)^{\alpha_{d}}/N)$, where $\alpha_{d}\ge 0$ is a constant that depends only on the dimension $d$. This kind of point set is widely used in many fields, such as colored noise, learning theory and finance ~\cite{Blue-Noise, Cristiano, phdpaper, rankla}, but this kind of point set also has some limitations. First of all, the convergence rate is strongly dependent on the dimensions. For high dimensional space, the convergence rate becomes slower, even worse than the convergence rate $O(N^{-\frac{1}{2}})$ obtained by classical Monte Carlo in some cases. While the classical Monte Carlo sampling method is independent of the dimensions, so it has strong applicability for especially large-dimensional data. Secondly, because the low discrepancy point set is a constructive sampling, it is difficult to deal with a large number of point sets generated by random sampling in real life, such as the processing of noise-contaminated signals. Therefore, it is necessary to explore a new sampling method that inherits the fast convergence characteristics of low discrepancy point sets in moderate dimensions and the strong applicability of Monte Carlo point sets to high dimensions.

In order to achieve a better approximation effect for multivariate integrals in higher dimensions, a large number of points must be sampled if we use the low discrepancy point sets. Which loses information about the moderate sample size. In other words, if it is required to find the optimal upper bound of the star discrepancy in the higher dimensional space under the moderate sample size, the low discrepancy point set is not applicable. Therefore, in recent years, many scholars have carried out a series of studies on this problem and tried to use the appropriate sampling point set. This study is called the search for the optimal pre-asymptotic star discrepancy bound. In 2001, Heinrich, Novak~\cite{Heinrich2001} et al. proved that the optimal pre-asymptotic upper bound can be obtained as long as the number of sampling points linearly depends on the dimension, and the upper bound is $c\cdot\frac{\sqrt{d}}{{\sqrt{N}}}$, where $c$ is an unknown positive constant. In 2011, Aistleitner~\cite{aistleitner2011covering} calculated that $c$ is at most 10. This constant has been improved to 2.53 ~\cite{Gnewuch2020} so far.  From a practical point of view, the following probabilistic pre-asymptotic upper bound ~\cite{aistleitner2014} can be obtained, that is,

\begin{equation}\label{sransAH}
D_N^{*}(Y)\leq 5.7\sqrt{4.9-\frac{\ln(1-q)}{d}}\frac{\sqrt{d}}{\sqrt{N}},
\end{equation}
holds with probability at least $q$ for the Monte Carlo point set $Y$.

Combined with the K-H inequality, it is easy to achieve the optimal probability approximation of the bounded variation function on the $d-$dimensional unit cube. Especially for high-dimensional data, this approximation overcomes some disadvantages of the traditional deterministic low discrepancy point sets, and gives a good approximation with a moderate sample size.

By optimizing the cover technique, the probability bound of \eqref{sransAH} is improved by ~\cite{Gnewuch2020}, which gives

\begin{equation}\label{sransGne}
D_N^{*}(Y)\leq 0.7729\sqrt{10.7042-\frac{\ln(1-q)}{d}}\frac{\sqrt{d}}{\sqrt{N}}.
\end{equation}
holds with probability at least $q$ for the Monte Carlo point set $Y$.

It is worth mentioning that the above pre-asymptotic star discrepancy bounds and the optimal probability approximation bounds are all obtained by using the Monte Carlo point set. Although the Monte Carlo method can get better approximation for high dimensional data, the convergence speed is slow, and the convergence speed is $O(N^{-\frac{1}{2}})$. How to improve this convergence speed and make the original approximation still applied to moderate sample size is a concern of scholars in recent years. Some scholars have noticed that the main reason for the slow convergence of the Monte Carlo method is the large variance of the sample mean function. Therefore, they have focused on the variance reduction technique and tried to use some new sampling methods to improve the approximation effect. For example, the probabilistic star discrepancy bound using Latin hypercube sampling ~\cite{Gnewuch2020} and the expected bound ~\cite{jittsamp,Doerr2} obtained by using jittered sampling. Jittered sampling is formed by isometric grid partition, and $[0,1]^{d}$ is divided into $m^d$ small sub-cubes $Q_{i},1\leq i\leq N,$ each with side length $\frac{1}{m},$ Figures \ref{fig1} and \ref{fig2} give $2$-dimensional and $3$-dimensional illustrations.

\begin{figure}[ht]%
\centering
\includegraphics[width=0.3\textwidth]{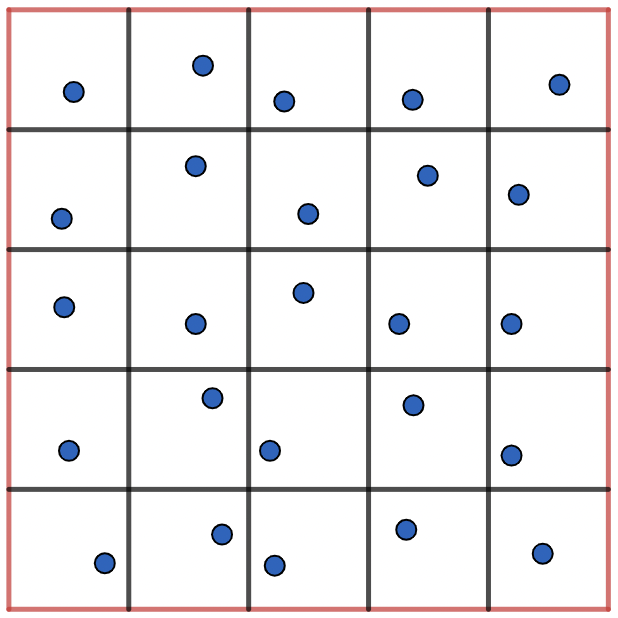}
\caption{Two-dimensional stratified sampling}\label{fig1}
\end{figure}

\begin{figure}[ht]%
\centering
\includegraphics[width=0.3\textwidth]{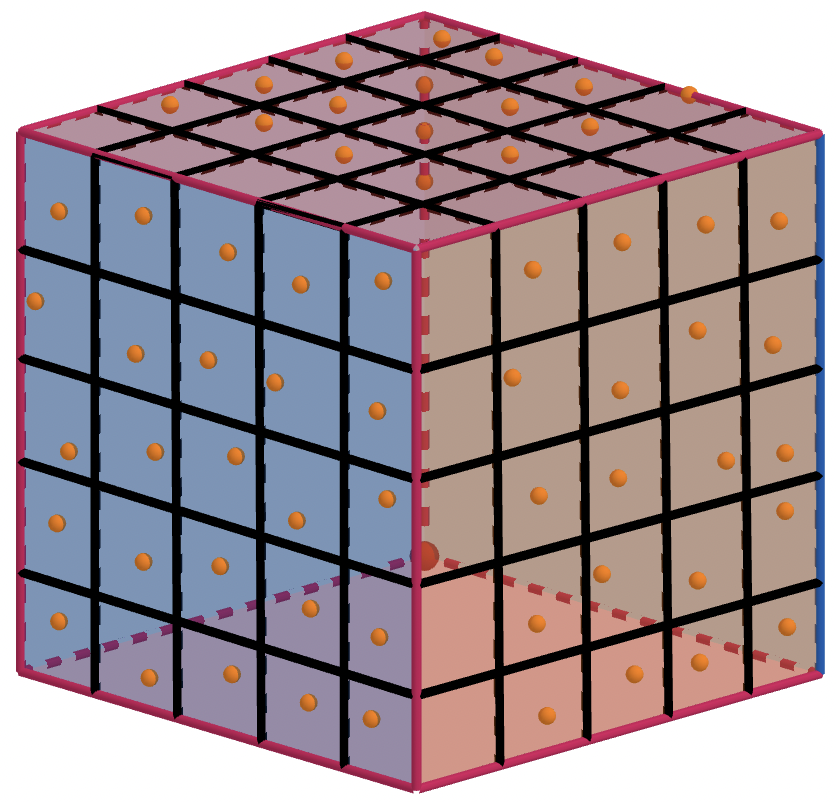}
\caption{Three-dimensional stratified sampling}\label{fig2}
\end{figure}

Moreover, although the classical jittered sampling inherits some advantages of stratified sampling, it is more demanding for the number of sampling points, i.e., $N=m^d$. The exponential dependence of the sampling number on $d$, which will inevitably lead to the curse of dimensionality, so it is difficult to obtain the pre-asymptotic star discrepancy bound. However, jittered sampling retains the value of theoretical and numerical research, such as, it can achieve a faster convergence rate than the Monte Carlo sampling in a relatively low-dimensional space. While Latin hypercube sampling overcomes the dimensionality problem and reduces the variance of the sample mean, the star discrepancy convergence order is still $O(N^{-\frac{1}{2}})$, see ~\cite{Gnewuch2020}.

If the problem is to reduce the error of the \eqref{K-H} under suitable sampling number, then for a class of bounded variation function spaces, there is

\begin{equation}\label{K-Hplusstarb1}
E_N(f)\leq 0.7729\sqrt{10.7042-\frac{\ln(1-q)}{d}}\frac{\sqrt{d}}{\sqrt{N}}\cdot V(f),\end{equation}
holds with probability at least $q$.

So how to improve the convergence speed of the approximation error in \eqref{K-Hplusstarb1} and preserve approximation information under the appropriate sampling number, based on the above introduction and discussion, there are two main ideas: one is to use the low discrepancy point sets, but note that the convergence order is $O((\ln N)^{\alpha_{d}}/N)$, where $\alpha_{d}\ge 0$ is a constant that depends on the dimension $d$. Generally for different point sets, $\alpha_{d}$ is at least $O(d)$, so for high-dimensional space, the number of samples required is particularly large; the other is to use the idea of stratified sampling, but for classical jittered sampling, the disadvantage is that $N=m^d$ sampling points are required, and for high-dimensional space, it faces the curse of dimensionality. Therefore, it is necessary to return to the general stratified sampling method. We try to use a sampling method that combines the advantages of low discrepancy point sets and stratified sampling, and to overcome the dimension problem to a certain extent.

The Hilbert space filling curve sampling method coincidentally inherits many advantages of general stratified sampling, and we describe it in detail below. This method was developed by A. B. Owen, Z. He et al. to study the variance of the sample mean function ~\cite{HO2016,HZ2019}. In our main result, we use Hilbert space filling curve sampling to obtain an improved upper bound of probability star discrepancy, and we give

\begin{equation}
D_{N}^{*}\left(X_{1}, X_{2}, \ldots, X_{N}\right)\leq \frac{c(d,q)}{N^{\frac{1}{2}+\frac{1}{2d}}},
\end{equation}
holds with probability at least $q$ for the Hilbert space filling curve sampling set $X$, where $c(d,q)$ is an explicit and computable constant depending on $d$ and $q$.

We can easily derive the probability star discrepancy bounds to the weighted form, and give the uniform probability approximation error in the weighted function space, that is, for a function in the weighted Sobolev space $F_{d,1,\gamma}$, the following 

\begin{equation}
\begin{aligned}
\sup_{f\in F_{d,1,\gamma},\|f\|_{F_{d,1,\gamma}}\leq 1}|\int_{[0,1]^{d}}f(x)dx-\frac{1}{N}\sum_{j=1}^{N}f(X_{j})|\leq\max_{\emptyset\neq u\subseteq I_{d}}\gamma_{u,d}\cdot \frac{c(|u|,q)}{N^{\frac{1}{2}+\frac{1}{2d}}}\end{aligned}
\end{equation}
holds with probability at least $\epsilon$ for the Hilbert space filling curve sampling set $X$.

The strong partition principle for the star discrepancy version is an open question mentioned in \cite{jittsamp, KP}, we use the probabilistic results to prove the following conclusion:

\begin{equation}
\mathbb{E}(D_{N}^{*}(P_{\Omega}))< \mathbb{E}(D_{N}^{*}(Y)),
\end{equation}
where $\Omega=\{\Omega_1,\Omega_2,\ldots,\Omega_N\}$ is an arbitrary equivolume partition and $P_{\Omega}$ is the corresponding stratified sampling set, and $Y=\{Y_{1}, Y_{2}, Y_{3}, \ldots, Y_{N}\}$ is simple random $d-$dimensional samples.

We also apply the main result to give explicit bounded expressions for a set of sampling inequalities in the general Hilbert space.

We prove:
\begin{align*}
(1-\frac{D_{N}^{*}(X)}{\gamma})\|f\|_{L^2}^2\leq \frac{1}{N}\sum_{j=1}^{N}f^{2}(x_j)\leq (\frac{D_{N}^{*}(X)}{\gamma}+1)\|f\|_{L^2}^2,
\end{align*}
holds with probability at least $1-\epsilon$($\epsilon$ the probability that $D_{N}^{*}(X)\leq \frac{c(d,\epsilon)}{N^{\frac{1}{2}+\frac{1}{2d}}}$ must hold), where $0<\gamma\leq 1$, and $f$ belongs to Hilbert space.

In the end, we use the variant K-H inequality, combined with the obtained probabilistic star discrepancy bound, to give an approximation error for the piecewise smooth function $f$ defined on a Borel convex subset of $[0,1]^d$.

We prove:

\begin{equation}
|\frac{\sum_{n=1}^{N}(f\cdot\mathbf{1}_{\Omega})(X_{n})}{N}-
\int_{\Omega}f(x)dx|\leq 2^{d}D_{N}^{*}(X_1,X_2,\ldots,X_N)\cdot V(f),
\end{equation}
holds with probability at least $q$, where $D_{N}^{*}(X_1,X_2,\ldots,X_N)$ is the star discrepancy of Hilbert space filling curve sampling $X$ and $q$ denotes the probability that the upper bound $D_{N}^{*}(X)\leq\frac{c(d,q)}{N^{\frac{1}{2}+\frac{1}{2d}}}$ must hold. Besides, 

\begin{equation}
V(f)=\sum_{u\subset\{1,2,\ldots,d\}}2^{d-|u|}\int_{[0,1]^{d}}
|\frac{\partial^{|u|}}{\partial x_{u}}f(x)|dx.
\end{equation}
The symbol $\frac{\partial^{|u|}}{\partial x_{n}}f(x)$ is the partial derivative of $f$ with respect to those components of $x$ with index in $u$.

The remainder of this paper is organised as follows. Section \ref{sec:notares} presents the preliminaries. Section \ref{mres:pfs} presents our main result. Section \ref{apps} gives some applications of the main result. Finally, in section \ref{discuss}, we conclude the paper with a short summary.

\section{Hilbert space filling curve sampling and some estimations}\label{sec:notares}

\textbf{1. Hilbert space filling curve sampling}

We first introduce Hilbert space filling curve sampling, which is abbreviated as HSFC sampling. We mainly adopt the definitions and symbols in \cite{HO2016,HZ2019}. HSFC sampling is actually a stratified sampling formed by a special partition method.

Let $a_i$ be the first $N=b^m$ points of the van der Corput sequence(van der Corput 1935) in base $b\ge 2, m=0,1,\ldots$. The integer $i-1\ge0$ is written in base $b$ as 

\begin{equation}
    i-1=\sum_{j=1}^{\infty}a_{ij}b^{j-1}
\end{equation}
for $a_{ij}\in \{0,\ldots,b-1\}$. Then, $a_i$ is defined by 

\begin{equation}
    a_i=\sum_{j=1}^{\infty}a_{ij}b^{-j}.
\end{equation}

The scrambled version of $a_1, a_2,\ldots, a_N$ is $x_1, x_2, \ldots, x_N$ written as 

\begin{equation}
    x_i=\sum_{j=1}^{\infty}x_{ij}b^{-j},
\end{equation}
where $x_{ij}$ are defined through random permutations of the $a_{ij}$. These permutations depend on $a_{ik},$ for $k<j.$ More precisely, $x_{i1}=\pi(a_{i1}),$ $x_{i2}=\pi_{a_{i1}}(a_{i2})$ and generally for $j\ge 2,$ 

\begin{equation}
    x_{ij}=\pi_{a_{i1\ldots a_i{j-1}}}(a_{ij}).
\end{equation}

Each random permutation is uniformly distributed over the $b!$ permutation of $\{0,\ldots,b-1\}$, and is mutually independent of the others. Thanks to the nice property of nested uniform scrambling, the data values in the scrambled sequence can be reordered such that

\begin{equation}
    x_i\sim U(I_i),
\end{equation}
independently with 
$$I_i=[\frac{i-1}{N},\frac{i}{N}]$$
for $i=1,2,\ldots, N=b^m$.

Let $$E_i=H(I_i):=\{H(x)|x\in I_i\},$$ where $H$ is a mapping.

Then

$$X_i=H(x_i)\sim U(E_i), i=1,2,3,\ldots N=b^m$$ is the corresponding stratified samples. Set $r_i$ be the diameter of $E_i$, according to the property of HSFC sampling, the following estimation holds,

\begin{equation}\label{diaest}
r_i\leq 2\sqrt{d+3}\cdot N^{-\frac{1}{d}}.
\end{equation}

\textbf{2. Minkowski content}

We use the definition of Minkowski content in \cite{ACV}, which provides convenience for analysing the boundary properties of the test set $B$ in \eqref{*-d}. For a set $\Omega\subset [0,1]^d$, define

\begin{equation}
\mathscr{M}(\partial \Omega)=\lim_{\epsilon\to 0}\frac{\lambda((\partial \Omega)_{\epsilon})}{2\epsilon}, \end{equation}
where $(\partial \Omega)_{\epsilon}=\{X\in \mathbb{R}^d|dist(x,\partial \Omega)\leq \epsilon\}$. If $\mathscr{M}(\partial \Omega)$ exists and is finite, then $\partial \Omega$ is said to admit $(d-1)-$dimensional Minkowski content. If $\Omega$ is a convex set, then it is easy to see that $\partial \Omega$ admits $(d-1)-$dimensional Minkowski content. Furthermore, $\mathscr{M}(\partial \Omega)\leq 2d$ as the outer surface area of a convex set in $[0,1]^d$ is bounded by the surface area of the unit cube $[0,1]^{d}$, which is $2d$.

\textbf{3. $\delta$-covers}

To discretise the star discrepancy, we use the definition of $\delta-$covers as in\cite{Doerr}.

\begin{defn}
For any $\delta\in (0,1]$, a finite set $\Gamma$ of points in $[0,1)^{d}$ is called a $\delta$–cover of $[0,1)^{d}$, if for every $y\in [0,1)^{d}$, there exist $x, z\in \Gamma\cup \{0\}$ such that $x\leq y\leq z$ and $\lambda([0, z])-\lambda([0, x])\leq \delta$. The number $\mathcal{N}(d,\delta)$ denotes the smallest cardinality of a $\delta$–cover of $[0,1)^{d}$.
\end{defn}

From \cite{Gnewuch2008Bracketing}, combined with Stirling's formula, the following estimate holds for $\mathcal{N}(d,\delta)$, that is, for any $d\ge 1$ and $\delta\in (0,1)$,

\begin{equation}\label{delcovbd}
    \mathcal{N}(d,\delta)\leq 2^d\cdot \frac{e^d}{\sqrt{2\pi d}}\cdot(\delta^{-1}+1)^{d}.
\end{equation}

Furthermore, the following lemma provides a convenient way to estimate the star discrepancy with $\delta-$covers.

\begin{lma}\label{cvdeldb}\cite{Doerr}
Let $P=\{p_{1},p_{2},\ldots,p_{N}\}\subset [0,1]^{d}$ and $\Gamma$ be $\delta-$covers, then

\begin{equation}
D_{N}^{*}(P)\leq D_{\Gamma}(P)+\delta,
\end{equation}
where

\begin{equation}
D_{\Gamma}(P):=\max_{x\in\Gamma}|\lambda([0,x])-\frac{\sum_{n=1}^{N} I_{[0, x]}\left(p_{n}\right)}{N}|.
\end{equation}

\end{lma}

\textbf{4. Bernstein inequality}

At the end of this section, we  will restate the Bernstein inequality which will be used in the estimation of star discrepancy bounds.

\begin{lma}\cite{FDX2007}\label{Binequ}
Let $Z_1,\ldots,Z_N$ be independent random variables with expected
values $\mathbb E (Z_j)=\mu_{j}$ and variances $\sigma_{j}^2$ for $j=1,\ldots,N$.  Assume  $|Z_j-\mu_j|\le C$(C is a constant) for each $j$ and set $\Sigma^2:=\sum_{j=1}^N\sigma_{j}^2$, then
for any $\lambda \ge 0$,
$$\mathbb P\left\{\Big|\sum_{j=1}^N [Z_j-\mu_j]\Big|\ge \lambda\right \}\le
2\exp\left(-\frac{\lambda^2}{2\Sigma^2+\frac{2}{3}C \lambda}\right).$$
\end{lma}

\section{Probabilistic star discrepancy bound for HSFC sampling}\label{mres:pfs}

\begin{thm}\label{convord1}
For integer number $b\ge 1$, $m\ge 1$ and $N=b^m$, then for the well defined $d-$dimensional stratified samples $X_i\sim U(E_i), i=1,2,\ldots, N=b^m$ in Section 2, we have

\begin{equation}
D_{N}^{*}\left(X_{1}, X_{2}, \ldots, X_{N}\right)\leq \frac{6d^{\frac{3}{4}}}{N^{\frac{1}{2}+\frac{1}{2d}}}\cdot\sqrt{d\ln(N+1)+c(d,q)}+\frac{2c(d,q)}{3N},
\end{equation}
holds with probability at least $q$,
where $$c(d,q)=\ln\frac{(2e)^{d}}{\sqrt{2\pi d}\cdot (1-q)}.$$
\end{thm}

\begin{rem}
 Theorem \ref{convord1} improves the convergence order of the probabilistic star discrepancy bound. It is improved from $O(N^{-\frac{1}{2}})$ to $O(N^{-\frac{1}{2}-\frac{1}{2d}}\cdot (\ln N)^{\frac{1}{2}})$. Which can be compared to that of \eqref{sransAH} and the part of the star discrepancy bound in \eqref{K-Hplusstarb1}. Secondly, the theorem is not as demanding as classical jittered sampling (i.e., it requires the number of sampling points to depend exponentially on the dimension, i.e., $N=m^d$. In fact the number of sampling points in HSFC sampling method does not depend on the dimensions). Therefor, it gives a better integral approximation error in the pre-asymptotic case, which can be obtained by using the K-H inequality.
\end{rem}
\begin{proof}
Let $A$ be a subset of $E_{i}$, then according to $X_{i}\sim U(E_{i}),1\leq i\leq N$, it follows that

\begin{equation}\label{unidistri}
\mathbb{P}(X_{i}\in A)=\frac{\lambda(A)}{\lambda(E_{i})}=N\lambda(A).
\end{equation}

Now, for an arbitrary $d-$dimensional rectangle $R=[0,x]\in[0,1]^{d}$ with diameter $\kappa$. When $x$ passes through the unit cube $[0,1)^{d}$, we can choose two points $y,z$ such that $y\leq x\leq z$ and $\lambda([0,z])-\lambda([0,y])\leq \frac{1}{N}$.  Let $R_{0}=[0,y]$ and $R_{1}=[0,z]$, then we have $$R_{0}\subseteq R\subset R_{1},$$ and

\begin{equation}\label{r1r0diff0}
\lambda(R_{1})-\lambda(R_{0})\leq \frac{1}{N}.
\end{equation}

We know that $R_{0}$ has a smaller diameter than $\kappa$, so we set it to $\kappa_{0}$, and $R_{1}$ has a larger diameter than $\kappa$, so we set it to $\kappa_{1}$. This forms a bracketing cover for the set $R$, and from \eqref{delcovbd} and \eqref{r1r0diff0}, we can derive the upper bound for the bracketing cover pair $(R_{0},R_{1})$. It has a cardinality at most $2^{d-1}\frac{e^{d}}{\sqrt{2\pi d}}(N+1)^{d}$. Therefore, from the lemma \ref{cvdeldb}, we obtain

\begin{equation}\label{sdbrela0}
D_{N}^{*}(Y_{1},Y_{2},
\ldots,Y_{N};R)\leq\max_{i=0,1}
D_{N}^{*}(Y_{1},Y_{2},
\ldots,Y_{N};R_{i})+\frac{1}{N}.
\end{equation}

For an anchored box $R$ in $[0,1]^{d}$, it is easy to check that $R$ is representable as a disjoint union of $E_{i}'s$ entirely contained in $R$ and the union of $l$ pieces which are the intersections of some $E_{j}'s$ and $R$, i.e,

\begin{equation}
R=\bigcup_{i\in I}E_{i}\cup\bigcup_{j\in J}(E_{j}\cap \partial R),
\end{equation}
where $I,J$ denote the index-sets.

By the definition of Minkowski content, for any $\sigma>2$, there exists $\epsilon_0$ such that $\lambda((\partial R)_{\epsilon})\leq \sigma \epsilon\mathscr{M}(\partial R)$ whenever $\epsilon\leq \epsilon_0$.

From \eqref{diaest}, the diameter for each $E_i$ is at most $2\sqrt{d+3}\cdot N^{-\frac{1}{d}}$. We can assume $N>(\frac{2\sqrt{d+3}}{\epsilon_0})^d$, then $2\sqrt{d+3}\cdot N^{-\frac{1}{d}}:=\epsilon<\epsilon_0$ and $\bigcup_{j\in J}(E_{j}\cap \partial R)\subseteq (\partial R)_{\epsilon}$; thus

$$
|J|\leq \frac{\lambda((\partial R)_{\epsilon})}{\lambda(E_i)}\leq \frac{\sigma \epsilon\mathscr{M}(\partial R)}{N^{-1}}=2\sqrt{d+3}\sigma \mathscr{M}(\partial R)N^{1-\frac{1}{d}}.
$$

Without loss of generality, we can set $\sigma=3$, and combined with the fact that $\mathscr{M}(\partial R)\leq 2d$, it follows that 

\begin{equation} \label{cardJ}
|J|\leq 12d\sqrt{d+3}\cdot N^{1-\frac{1}{d}}.
\end{equation}

The same argument \eqref{cardJ} applies to test sets $R_0$ and $R_1$.

For $R_0$ or $R_1$, set 

\begin{equation}
D_{N}^{*}(X_{1},X_{2},
\ldots,X_{N};R')=\max_{i=0,1}
D_{N}^{*}(X_{1},X_{2},
\ldots,X_{N};R_{i}).
\end{equation}

$R'$ is also a test rectangle, which can be divided into two parts

\begin{equation}
R'=\bigcup_{k\in K}E_{k}\cup\bigcup_{l\in L}(E_{l}\cap R'),
\end{equation}
and the cardinality of $R^{'}\subset[0,1)^{d}$ is at most $2^{d-1}\frac{e^{d}}{\sqrt{2\pi d}}(N+1)^{d}$ according to the $\delta-cover$ estimate. 

Let

\begin{equation}
T=\bigcup_{l\in L}(E_{l}\cap R'), |L|=|\{1,2,\ldots,l\}|.
\end{equation}

If we define new random variables $\chi_{j}, 1\leq j\leq l$ as follows

\begin{equation}
    \chi_{l}=\left\{
\begin{aligned}
&1, X_{l}\in E_{l}\cap R'\\&
0, otherwise,
\end{aligned}
\right.
\end{equation}
then from the above discussions, we have

\begin{equation}\label{murmut0311}
\begin{aligned}
&N\cdot D_{N}^{*}\left(X_{1}, X_{2}, \ldots, X_{N};R'\right)\\&=N\cdot D_{N}^{*}\left(X_{1}, X_{2}, \ldots, X_{N};T\right)\\&=|\sum_{l=1}^{|L|}\chi_{l}-N(\sum_{l=1}^{|L|}\lambda(E_{l}\cap R'))|.
\end{aligned}
\end{equation}

Since

\begin{equation}
\mathbb{P}(\chi_{l}=1)=\frac{\lambda(E_{l}\cap R')}{\lambda(E_{l})}=N\cdot\lambda(E_{l}\cap R'),
\end{equation}
hence

\begin{equation}\label{echi0313}
    \mathbf{E}(\chi_{l})=N\cdot\lambda(E_{l}\cap R').
\end{equation}

Thus, from \eqref{murmut0311} and \eqref{echi0313}, we obtain

\begin{equation}\label{Rprime00314}
N\cdot D_{N}^{*}\left(X_{1}, X_{2}, \ldots, X_{N};R'\right)=|\sum_{l=1}^{|L|}(\chi_{l}-\mathbb{E}(\chi_{l}))|.
\end{equation}

Let $\sigma_{l}^{2}=\mathbb{E}(\chi_{l}-\mathbb{E}(\chi_{l}))^{2}$ and $\Sigma^{2}=(\sum_{l=1}^{|L|}\sigma_{l}^{2})^{\frac{1}{2}}$, then we have

\begin{equation}
    \Sigma^{2}\leq |L|\leq 12d\sqrt{d+3}\cdot N^{1-\frac{1}{d}}.
\end{equation}

Therefore, from Lemma \ref{Binequ}, for every $R'$, we have,

\begin{equation}\label{pchi}
\mathbb{P}\left
(\Big|\sum_{l=1}^{|L|}(\chi_{l}-\mathbb{E}(\chi_{l}))\Big|>\lambda\right)\leq
2\cdot\exp(-\frac{\lambda^{2}}{24d\sqrt{d+3}\cdot N^{1-\frac{1}{d}}+\frac{2\lambda}{3}}).
\end{equation}

Let

\begin{equation}\label{scrb}
\mathscr{B}=\bigcup_{R'}\left(\Big|\sum_{l=1}^{|L|}
(\chi_{l}-\mathbb{E}(\chi_{l}))|>\lambda\right).
\end{equation}

Then using covering numbers, we have

\begin{equation}\label{pb00}
\mathbb{P}(\mathscr{B})\leq (2e)^{d}\cdot\frac{1}{\sqrt{2\pi d}}\cdot(N+1)^{d}\cdot
\exp(-\frac{\lambda^{2}}{24d\sqrt{d+3}\cdot N^{1-\frac{1}{d}}+\frac{2\lambda}{3}}).
\end{equation}

Let $A(d,q,N)=d\ln(2e)+d\ln(N+1)-\frac{\ln(2\pi d)}{2}-\ln(1-q)$, and we choose

\begin{equation}
\lambda=\sqrt{24d\sqrt{d+3}\cdot A(d,q,N)+\frac{A^2(d,q,N)}{9N^{1-\frac{1}{d}}}}\cdot N^{\frac{1}{2}-\frac{1}{2d}}+\frac{A(d,q,N)}{3}.
\end{equation}

By substituting $\lambda$ into \eqref{pb00}, we have

\begin{equation}
\mathbb{P}(\mathscr{B})\leq 1-q.
\end{equation}

Combining above and \eqref{Rprime00314}, we obtain

\begin{equation}
\begin{aligned}
&D_{N}^{*}\left(X_{1}, X_{2}, \ldots, X_{N};R'\right)\\&\leq \frac{\sqrt{24d\sqrt{d+3}\cdot A(d,q,N)+\frac{A^2(d,q,N)}{9N^{1-\frac{1}{d}}}}}
{N^{\frac{1}{2}+\frac{1}{2d}}}
+\frac{A(d,q,N)}{3N}
\end{aligned}
\end{equation}

holds with probability at least $q$.

Thus, obviously, we have

\begin{equation}\label{maxsdb0}
\begin{aligned}
&\max_{i=0,1}
D_{N}^{*}(X_{1},X_{2},
\ldots,X_{N};R_{i})\\&\leq \frac{\sqrt{24d\sqrt{d+3}\cdot A(d,q,N)+\frac{A^2(d,q,N)}{9N^{1-\frac{1}{d}}}}}
{N^{\frac{1}{2}+\frac{1}{2d}}}
+\frac{A(d,q,N)}{3N}\\&\leq \frac{\sqrt{24d\sqrt{d+3}\cdot A(d,q,N)}}
{N^{\frac{1}{2}+\frac{1}{2d}}}
+\frac{2A(d,q,N)}{3N}
\end{aligned}
\end{equation}
holds with probability at least $q$.

\begin{equation}
    \begin{aligned}
    A(d,q,N)&=\ln\frac{(2e)^{d}}{\sqrt{2\pi d}\cdot (1-q)}+d\ln(N+1)\\&=c(d,q)+d\ln(N+1),
    \end{aligned}
\end{equation}
where

\begin{equation}
    c(d,q)=\ln\frac{(2e)^{d}}{\sqrt{2\pi d}\cdot (1-q)}.
\end{equation}

The proof is completed.
\end{proof}

\section{Applications of the main result}\label{apps}

The following will give three applications of the theorem ~\ref{convord1}. One is the application of uniform integral approximation in weighted function space, and the other is to prove the strong partition principle of star discrepancy version. The strong partition principle is the problem proposed by ~\cite{KP,KP2}, which has been perfectly solved for the $L_2-$discrepancy, but remains to be solved for the star discrepancy version. The third is the sampling theorem in general Hilbert space. Finally, it is applied to the integral approximation problem on the Borel convex subset in $[0,1]^d$.

\subsection{Uniform integral approximation for functions in weighted function space}

Many high-dimensional problems in practical applications have low effective dimension ~\cite{Kuo2005}, that is, they have different weights for the function values of different components. Therefore, the problem is abstracted as finding the uniform integral approximation error in weighted Sobolev space.

Let $F_{d,1}$ be a Sobolev space, for functions $f\in F_{d,1} $, $f$ is differentiable for any variable and has a finite $L_{1}-$module for its first-order differential. For $d>1$, the norm in $F_{d,1}$ is defined as 

$$
\|f\|_{F_{d,1}}=\|D^{\vec{1}}f\|_{L_{1}([0,1]^{d})}
=\int_{[0,1]^{d}}|D^{\vec{1}}f(x)|dx,
$$
where $\vec{1}=[1,1,\ldots,1]$ and $D^{\vec{1}}=\partial^{d}/\partial x_{1}\ldots\partial x_{d}$.

Then for weighted Sobolev space $F_{d,1,\gamma}$, its norm is 

\begin{equation}\label{fd1g}
\|f\|_{F_{d,1,\gamma}}=\sum_{u\subseteq I_{d}}\gamma_{u,d}
\int_{[0,1]^{|u|}}|\frac{\partial^{|u|}}{\partial x_{u}}f(x_{u},1)|dx_{u}.\end{equation}

Considering the problem of function approximation in $F_{d,1,\gamma}$ space. The sample mean method can still be used, for

$$
I(f)=\int_{[0,1]^d}f(x)dx,
$$
and sample mean function

$$
\tilde{I}(f,\mathbf{P})=\frac{1}{N}\sum_{i=1}^{N}f(x_{i}).
$$

We consider the worst-case error 
\begin{equation*}
  E_N(f)=\left|\int_{[0,1]^{d}} f(x) d x-\frac{1}{N} \sum_{t \in P_{N, d}} f(t)\right|.
\end{equation*}

Then according to Hlawka and Zaremba identity \cite{Ems}, we have

$$
e(E_N(f))=\sup_{f\in F_{d,1,\gamma},\|f\|_{F_{d,1,\gamma}}\leq 1}|I(f)-\tilde{I}(f,\mathbf{P})|
=D_{N,\gamma}^{*}(t_{1},t_{2},\ldots,t_{N}).
$$

For uniform integral approximation in weighted Sobolev spaces, we have the following theorem:

\begin{thm}\label{sobo2}

Let $f\in F_{d,1,\gamma}$ be functions in Sobolev sapce. Let integer $b\ge 1$, $m\ge 1$ and $N=b^m$. Let $b,\lambda,\lambda_{0},c$ be some integers, $\lambda_{0}$ be constants such that $b\lambda^{2}\leq e^{2\lambda^{2}}$ holds for all $\lambda\ge\lambda_{0}$, and $c=\max\{2,b,\lambda_{0},\frac{1}{\log_{2}(2-\epsilon)}\}$. Then for $d-$dimensional Hilbert space filling curve samples $X_i\sim U(E_i), i=1,2,\ldots, N=b^m$, we have

\begin{equation}
\begin{aligned}
\sup_{f\in F_{d,1,\gamma},\|f\|_{F_{d,1,\gamma}}\leq 1}&|\int_{[0,1]^{d}}f(x)dx-\frac{1}{N}\sum_{j=1}^{N}f(X_{j})|\\&\leq\max_{\emptyset\neq u\subseteq I_{d}}\gamma_{u,d}[\frac{6|u|^{\frac{3}{4}}}{N^{\frac{1}{2}+\frac{1}{2|u|}}}\cdot\sqrt{|u|\ln(N+1)+c(|u|,\epsilon)}+\frac{2c(|u|,\epsilon)}{3N}]
\end{aligned}
\end{equation}
holds with probability at least $\epsilon$, where 
$$c(|u|,\epsilon)=\ln\frac{(2e)^{|u|}}{\sqrt{2\pi |u|}\cdot (1-\epsilon)}.$$
\end{thm}

\begin{proof}
In theorem \ref{convord1}, we choose the probability $q=\epsilon=1-(b\lambda^{2}e^{-2\lambda^{2}})^{d}$. Which holds for some positive constant $b$ and for all $\lambda\ge\max\{1,b,\lambda_{0}\}$, where $\lambda_{0}$ is constant such that $b\lambda^{2}\leq e^{2\lambda^{2}}$ holds for all $\lambda\ge\lambda_{0}$, and we choose $$\lambda =c\max\{1,\sqrt{(\ln d)/(\ln 2)}\},$$ and $c=\max\{2,b,\lambda_{0}\}$.

Let $$c(|u|,\epsilon)=\ln\frac{(2e)^{|u|}}{\sqrt{2\pi |u|}\cdot (1-\epsilon)}.$$

For a given number of sample points $N$ and dimension $d$, we consider the following set

\begin{align*}
A_{d}:=\{\mathbf{P}_{N,d}\subset[0,1]^{d}:
&D_{N}(\mathbf{P}_{N,d}(u))\leq[\frac{6|u|^{\frac{3}{4}}}{N^{\frac{1}{2}+\frac{1}{2|u|}}}\cdot\sqrt{|u|\ln(N+1)+c(|u|,\epsilon)}+\frac{2c(|u|,\epsilon)}{3N}], \\&
\forall u\subseteq I_{d},u\neq \emptyset\},
\end{align*}
where $\mathbf{P}_{N,d}(u):=\{X_{1}(u),\ldots,X_{N}(u)\}$. In addition, for $u\subseteq I_{d}, u\neq\emptyset$, we define

\begin{align*}
A_{u,d}:=\{\mathbf{P}_{N,d}\subset[0,1)^{d}:&D_{N}(\mathbf{P}_{N,d}(u))\leq[\frac{6|u|^{\frac{3}{4}}}{N^{\frac{1}{2}+\frac{1}{2|u|}}}\cdot\sqrt{|u|\ln(N+1)+c(|u|,\epsilon)}\\&+\frac{2c(|u|,\epsilon)}{3N}]\}.
\end{align*}

Then we have 

\begin{align*}
A_{d}=\bigcap_{\emptyset\neq u\subseteq I_{d}}A_{u,d}.
\end{align*}

Hence,

$$
\begin{aligned}
\mathbb{P}(A_{d})&=\mathbb{P}(\bigcap_{\emptyset\neq u\subseteq I_{d}}A_{u,d})=1-\mathbb{P}(\bigcup_{\emptyset\neq u\subseteq I_{d}}A_{u,d}^{c})\\&\ge1-\sum_{\emptyset\neq u\subseteq I_{d}}\mathbb{P}(A_{u,d}^{c})>1-\sum_{\emptyset\neq u\subseteq I_{d}}(b\lambda^{2}e^{-2\lambda^{2}})^{|u|}
\\&=1-\sum_{u=1}^{d}\dbinom{d}{u}(b\lambda^{2}e^{-2\lambda^{2}})^{u}
=2-(1+b\lambda^{2}e^{-2\lambda^{2}})^{d}.
\end{aligned}
$$

According to $\lambda =c\max\{1,\sqrt{\frac{\ln d}{\ln 2}}\}$ and $c =\max\{2,b,\lambda_{0}\}$, for all $d\ge 2$ and $x=\frac{c^{2}}{\ln 2}>5$, we have $x^{2}\leq 2^{x}\leq d^{x}$ and $\ln d\leq d^{x-1}$. Thus, $x^{2}\ln d\leq d^{2x-1}$, then

$$
\frac{c^{3}\ln d}{(\ln 2)d^{\frac{2c^{2}}{\ln2}}}\leq\frac{\ln 2}{cd}.
$$

Based on this inequality, we obtain a formula that holds for all $d\ge 2$

$$
\begin{aligned}
\mathbb{P}(A_{d})&>2-(1+b\lambda^{2}e^{-2\lambda^{2}})^{d}
\ge2-(1+\frac{c^{3}\ln d}{(\ln 2)d^{\frac{2c^{2}}{\ln2}}})^{d}
\\&\ge2-(1+\frac{\ln 2}{cd})^{d}>2-e^{\frac{\ln 2}{c}}=2-2^{1/c}\ge\epsilon=q\in(0,1).
\end{aligned}
$$

Thus for every $\emptyset\neq u\subseteq I_{d}$, we get

$$
D_{N,\gamma}\left(t_{1}, t_{2}, \ldots t_{N}\right)\leq\max_{\emptyset\neq u\subseteq I_{d}}\gamma_{u,d}[\frac{6|u|^{\frac{3}{4}}}{N^{\frac{1}{2}+\frac{1}{2|u|}}}\cdot\sqrt{|u|\ln(N+1)+c(|u|,\epsilon)}+\frac{2c(|u|,\epsilon)}{3N}]
$$
holds with probability at least $\epsilon$.

The proof is completed.
\end{proof}

\subsection{Strong partition principle for star discrepancy}

For stratified sampling, ~\cite{Doerr2} proves that the classical jittered sampling can obtain a better expected star discrepancy bound than the traditional Monte Carlo sampling. Which means that the jittered sampling is the optimization of Monte Carlo sampling when the number of sampling points is $N=m^d$. Then the problem is how to directly compare the size of the expected star discrepancy under different sampling methods, rather than the optimization of the bound. It refers to how to prove the partition principle of star discrepancy. That is, the expected star discrepancy of the HSFC stratified sampling( here is no longer limited to the jittered case) must be smaller than the expected star discrepancy of simple random sampling.

\begin{thm}\label{uniformnoise0}
For integers $b\ge 1$, $m\ge 1$ and $N=b^m$. For well-defined Hilbert space filling curve samples $X_i\sim U(E_i), i=1,2,\ldots, N=b^m$, and simple random $d-$dimensional samples $Y=\{Y_{1}, Y_{2}, Y_{3}, \ldots, Y_{N}\}$, then

\begin{equation}\label{formula31}
\mathbb{E}(D_{N}^{*}(X))< \mathbb{E}(D_{N}^{*}(Y)).
\end{equation}

\end{thm}

\begin{proof}
For arbitrary test set $R=[0,x)$, we unify a label $W=\{W_1,W_2,\ldots,W_N\}$. Which is the sampling point sets formed by different sampling models, and we consider the following discrepancy function,

\begin{equation}\label{dispfunc1}
\Delta_{\mathscr{P}}(x)=\frac{1}{N}
\sum_{n=1}^{N}\mathbf{1}_{R}(W_n)-\lambda(R).
\end{equation}

For an equivolume partition $\Omega=\{\Omega_1,\Omega_2,\ldots,\Omega_N\}$, we divide the test set $R$ into two parts, one being the disjoint union of $\Omega_{i}$ completely contained by $R$, and the other being the union of the remaining parts which are the intersections of some $\Omega_{j}$ and $R$, i.e.,

\begin{equation}\label{Rtp1}
R=\bigcup_{i\in I_0}\Omega_{i}\cup\bigcup_{j\in J_0}(\Omega_{j}\cap R),
\end{equation}
where $I_0,J_0$ are two index sets.

We set $$T=\bigcup_{j\in J_0}(\Omega_{j}\cap R).$$

Furthermore, for the equivolume partition $\Omega=\{\Omega_1, \Omega_2, \ldots, \Omega_N\}$ of $[0,1]^{d}$ and the corresponding stratified sampling set $P_{\Omega}$, we have

\begin{equation}\label{for4d4}
\begin{aligned}
\text{Var}(\sum_{n=1}^{N}\mathbf{1}_{R}(P_{\Omega}))&=\sum_{i=1}^{N}\frac{|\Omega_i\cap[0,x]|}{|\Omega_i|}(1-\frac{|\Omega_i\cap[0,x]|}{|\Omega_i|})\\&=N|[0,x]|-N^2\sum_{i=1}^{N}(|\Omega_i\cap[0,x]|)^2.
\end{aligned}
\end{equation}

For sampling sets $Y=\{Y_1,Y_2,\ldots,Y_N\}$ and $X=\{X_1,X_2,\ldots,X_N\}$, we have
\begin{equation}
\text{Var}(\sum_{n=1}^{N}\mathbf{1}_{R}(X))=N|[0,x]|-N^2\sum_{i=1}^{N}(|\Omega^{*}_{|,i}\cap[0,x]|)^2,
\end{equation}
and 
\begin{equation}
\text{Var}(\sum_{n=1}^{N}\mathbf{1}_{R}(Y))=N|[0,x]|-N^2|[0,x]|^2.
\end{equation}

Hence, we have

\begin{equation}\label{varcomp}
\text{Var}(\sum_{n=1}^{N}\mathbf{1}_{R}(X))\leq \text{Var}(\sum_{n=1}^{N}\mathbf{1}_{R}(Y)).
\end{equation}

Now we exclude the case of equality in \eqref{varcomp}, which means that the following formula holds,

\begin{equation}
N|\Omega^{*}_{|,i}\cap[0,x]|=|[0,x]|,
\end{equation}
for $i=1,2,\ldots,N$ and for almost all $x\in [0,1]^{d}$.

Hence, 

\begin{equation}
\int_{[0,x]}\mathbf{1}_{\Omega^{*}_{|,i}}(y)dy=\int_{[0,x]}\frac{1}{N}dy,
\end{equation}
for almost all $x\in[0,1]^{d}$ and all $i=1,2,\ldots,N$, which implies $\mathbf{1}_{\Omega^{*}_{|,i}}=\frac{1}{N}$. This is not impossible for $N\ge 2.$

For set $T$, we have

\begin{equation}
\text{Var}(\frac{1}{N}
\sum_{n=1}^{N}\mathbf{1}_{T}(X_n))=\sum_{i=1}^{|J_0|}\frac{|\Omega^{*}_{i,|}\cap[0,x]|}{|\Omega^{*}_{i,|}|}(1-\frac{|\Omega^{*}_{i,|}\cap[0,x]|}{|\Omega^{*}_{i,|}|}).
\end{equation}

The same analysis for the set $T$ as $R$, we have

\begin{equation}\label{varcom000}
   \text{Var}(\frac{1}{N}
\sum_{n=1}^{N}\mathbf{1}_{T}(X_n))<\text{Var}(\frac{1}{N}
\sum_{n=1}^{N}\mathbf{1}_{T}(Y_n)).
\end{equation}

For the test set $R=[0,x)$, we choose $R_0=[0,y)$ and $R_1=[0,z)$ such that $y\leq x\leq z$ and $\lambda(R_1)-\lambda(R_0)\leq \frac{1}{N}$, then ($R_0$,$R_1$) form the $\frac{1}{N}-$covers. For $R_0$ and $R_1$, we can divide them into two parts as we did for \eqref{Rtp1}. Let $$T_0=\bigcup_{j\in J_0}(\Omega_{j}\cap R_0),$$ and $$T_1=\bigcup_{j\in J_0}(\Omega_{j}\cap R_1).$$  We have the same result for $T_0$ and $T_1$. To unify the two cases $T_0$ and $T_1$ (because $T_0$ and $T_1$ are generated from two test sets with the same cardinality, and the cardinality is the covering numbers), we consider a set $R'$ which can be divided into two parts

\begin{equation}\label{Rpr1}
R'=\bigcup_{k\in K}\Omega_{k}\cup\bigcup_{l\in L}(\Omega_{l}\cap R'),
\end{equation}
where $K,L$ are two index sets. In addition, we set the cardinality of $R^{'}\subset[0,1)^{d}$ to be at most $2^{d-1}\frac{e^{d}}{\sqrt{2\pi d}}(N+1)^{d}$ (the $\delta-$covering numbers, where we choose $\delta=\frac{1}{N}$), and we let

$$
T'=\bigcup_{l\in L}(\Omega_{l}\cap R').
$$

We define new random variables $\chi_{j}, 1\leq j\leq |L|$, as follows

$$
    \chi_{j}=\left\{
\begin{aligned}
&1, W_{j}\in \Omega_{j}\cap R',\\&
0, otherwise.
\end{aligned}
\right.
$$
Then,

\begin{equation}\label{murmut0}
\begin{aligned}
&N\cdot D^{*}_{N}\left(W_{1}, W_{2}, \ldots, W_{N};R'\right)\\&=N\cdot D^{*}_{N}\left(W_{1}, W_{2}, \ldots, W_{N};T'\right)\\&=|\sum_{j=1}^{|L|}\chi_{j}-N(\sum_{j=1}^{|L|}\lambda(\Omega_{j}\cap T'))|.
\end{aligned}
\end{equation}

Since

$$
\mathbb{P}(\chi_{j}=1)=\frac{\lambda(\Omega_{j}\cap T')}{\lambda(\Omega_{j})}=N\cdot\lambda(\Omega_{j}\cap T'),
$$
we get

\begin{equation}\label{echi0}
    \mathbf{E}(\chi_{j})=N\cdot\lambda(\Omega_{j}\cap T').
\end{equation}

Thus, from \eqref{murmut0} and \eqref{echi0}, we obtain

\begin{equation}\label{Rprime00}
N\cdot D^{*}_{N}\left(W_{1}, W_{2}, \ldots, W_{N};R'\right)=|\sum_{j=1}^{|L|}(\chi_{j}-\mathbb{E}(\chi_{j}))|.
\end{equation}

Let $$\sigma_{j}^{2}=\mathbb{E}(\chi_{j}-\mathbb{E}(\chi_{j}))^{2}, \Sigma=(\sum_{j=1}^{|L|}\sigma_{j}^{2})^{\frac{1}{2}}.$$

Therefore, from Lemma \ref{Binequ}, for every $R'$, we have,

$$
\mathbb{P}\left
(\Big|\sum_{j=1}^{|L|}(\chi_{j}-\mathbb{E}(\chi_{j}))\Big|>\lambda\right)\leq
2\cdot\exp(-\frac{\lambda^{2}}{2\Sigma^{2}+\frac{2\lambda}{3}}).
$$

Let $\mathscr{B}=\bigcup\limits_{R'}\left(\Big|\sum_{j=1}^{|L|}(\chi_{j}-\mathbb{E}(\chi_{j}))|>\lambda\right),$ then using $\delta-$cover numbers, we have

\begin{equation}\label{pb000}
\mathbb{P}(\mathscr{B})\leq (2e)^{d}\cdot\frac{1}{\sqrt{2\pi d}}\cdot(N+1)^{d}\cdot
\exp(-\frac{\lambda^{2}}{2\Sigma^{2}+\frac{2\lambda}{3}}).
\end{equation}

Combining with \eqref{Rprime00}, we get

\begin{equation}\label{punrp1}
\begin{aligned}
&\mathbb{P}\Big(\bigcup_{R'}\left(N\cdot D_{N}^{*}\left(W_{1}, W_{2}, \ldots, W_{N};R'\right)>\lambda\right)\Big)\\&\leq (2e)^{d}\cdot\frac{1}{\sqrt{2\pi d}}\cdot(N+1)^{d}\cdot
\exp(-\frac{\lambda^{2}}{2\Sigma^{2}+\frac{2\lambda}{3}}).
\end{aligned}
\end{equation}

For point sets $Y$ and $X$, if we let $$\Sigma_{0}^2=\text{Var}(\sum_{n=1}^{N}\mathbf{1}_{T'}(X_n)), \Sigma_{1}^2=\text{Var}(\sum_{n=1}^{N}\mathbf{1}_{T'}(Y_n)).$$

Then \eqref{varcom000} implies

$$\Sigma_{0}^2<\Sigma_{1}^2.$$

Besides, as \eqref{punrp1}, we have

\begin{equation}\label{purpy2}
\begin{aligned}
&\mathbb{P}\Big(\bigcup_{R'}\left(N\cdot D_{N}^{*}\left(X_{1}, X_{2}, \ldots, X_{N};R'\right)>\lambda\right)\Big)\\&\leq (2e)^{d}\cdot\frac{1}{\sqrt{2\pi d}}\cdot(N+1)^{d}\cdot
\exp(-\frac{\lambda^{2}}{2\Sigma_0^{2}+\frac{2\lambda}{3}}),
\end{aligned}
\end{equation}
and

\begin{align*}
&\mathbb{P}\Big(\bigcup_{R'}\left(N\cdot D_{N}^{*}\left(Y_{1}, Y_{2}, \ldots, Y_{N};R'\right)>\lambda\right)\Big)\\&\leq (2e)^{d}\cdot\frac{1}{\sqrt{2\pi d}}\cdot(N+1)^{d}\cdot
\exp(-\frac{\lambda^{2}}{2\Sigma_1^{2}+\frac{2\lambda}{3}}).
\end{align*}

Suppose $A(d,q,N)=d\ln(2e)+d\ln(N+1)-\frac{\ln(2\pi d)}{2}-\ln(1-q)$, and we substitute

$$
\lambda=\sqrt{2\Sigma_0^2\cdot A(d,q,N)+\frac{A^2(d,q,N)}{9}}+\frac{A(d,q,N)}{3}
$$
into \eqref{purpy2}, then we have

\begin{equation}\label{punionrp}
\mathbb{P}\Big(\bigcup_{R'}\left(N\cdot D_{N}^{*}\left(X_{1}, X_{2}, \ldots, X_{N};R'\right)>\lambda\right)\Big)\leq 1-q.
\end{equation}

Therefore, from \eqref{punionrp}, it can be easily verified that

\begin{equation}\label{maxdnstar}
\max_{R_i,i=0,1}
D_{N}^{*}(X_{1},X_{2},\ldots,X_{N};R_{i})\leq \frac{\sqrt{2\Sigma_0^2\cdot A(d,q,N)+\frac{A^2(d,q,N)}{9}}}{N}+\frac{A(d,q,N)}{3N}
\end{equation}
holds with probability at least $q$.

Combined with $\delta-$covering numbers(where $\delta=\frac{1}{N}$), we get,

\begin{equation}\label{dnsy1}
\begin{aligned}
    D_{N}^{*}(X)&\leq \frac{\sqrt{2\Sigma_0^2\cdot A(d,q,N)+\frac{A^2(d,q,N)}{9}}}{N}+\frac{A(d,q,N)+3}{3N}\\&
    \leq (\sqrt{2}\cdot\Sigma_0+1)\frac{A(d,q,N)}{N}
\end{aligned}
\end{equation}
holds with probability at least $q$. The last inequality in \eqref{dnsy1} holds because $A(d,q,N)\ge 3$ holds for all $q\in (0,1)$.

Same analysis with point set $Y$, we have

\begin{equation}\label{DNsgm1}
\begin{aligned}
    D_{N}^{*}(Y)&\leq \frac{\sqrt{2\Sigma_1^2\cdot A(d,q,N)+\frac{A^2(d,q,N)}{9}}}{N}+\frac{A(d,q,N)+3}{3N}\\&
    \leq (\sqrt{2}\cdot\Sigma_1+1)\frac{A(d,q,N)}{N}
\end{aligned}
\end{equation}
holds with probability at least $q$.

Now we fix a probability value $q_0\in(0,1)$ in \eqref{dnsy1}, i.e. we assume that \eqref{dnsy1} holds with probability exactly $q_0$, where $q_0\in [q,1)$. If we choose this $q_0$ in \eqref{DNsgm1}, we have

$$
    D_{N}^{*}(Y)\leq (\sqrt{2}\cdot\Sigma_1+1)\frac{A(d,q_0,N)}{N},
$$
holds with probability $q_0.$

Therefore from $\Sigma_0<\Sigma_1$, we obtain,

\begin{equation}\label{dnsxq1}
     D_{N}^{*}(Y)\leq (\sqrt{2}\cdot\Sigma_0+1)\frac{A(d,q_0,N)}{N}
\end{equation}
holds with probability $q_1,$ where $0<q_1< q_0$.

We use the following fact to estimate the expected star discrepancy

\begin{equation}\label{EDNSX}
    \mathbb{E}[D^{*}_{N}(W)]=\int_{0}^{1}\mathbb{P}(D^{*}_{N}(W)\ge t)dt,
\end{equation}
where $D^{*}_{N}(W)$ denotes the star discrepancy of point set $W$.

By substituting $q_0$ into \eqref{dnsy1}, we obtain

\begin{equation}\label{DNsY}
D_{N}^{*}\left(X\right)\leq (\sqrt{2}\cdot\Sigma_0+1)\frac{A(d,q_0,N)}{N}
\end{equation}
holds with probability $q_0$. Then \eqref{DNsY} is equivalent to

$$
    \mathbb{P}\big(D_{N}^{*}\left(X\right)\ge (\sqrt{2}\cdot\Sigma_0+1)\frac{A(d,q_0,N)}{N}\big)=1-q_0.
$$

Now releasing $q_0$ and let

\begin{equation}\label{tsig0}
t=(\sqrt{2}\cdot\Sigma_0+1)\frac{A(d,q_0,N)}{N},
\end{equation}

\begin{equation}\label{c0sig0}
C_0(\Sigma_0,N)=\frac{\sqrt{2}\cdot\Sigma_0+1}{N},
\end{equation}
and
\begin{equation}\label{c1sig0}
C_1(d,\Sigma_0,N)=\frac{\sqrt{2}\cdot\Sigma_0+1}{N}\cdot(d\ln(2e)+d\ln(N+1)-\frac{\ln(2\pi d)}{2}).
\end{equation}

Then

\begin{equation}\label{tc1sig0}
    t=C_1(d,\Sigma_0,N)-C_0(\Sigma_0,N)\ln (1-q_0).
\end{equation}

Thus from \eqref{EDNSX} and $q_0\in[q,1)$, we have

\begin{equation}\label{ednstarz}
\begin{aligned}
&\mathbb{E}[D^{*}_{N}(X)]=\int_{0}^{1}\mathbb{P}(D^{*}_{N}(X)\ge t)dt\\&=\int_{1-e^{\frac{C_1(d,\Sigma_0,N)}{C_0(\Sigma_0,N)}}}^{1-e^{\frac{C_1(d,\Sigma_0,N)-1}{C_0(\Sigma_0,N)}}}\mathbb{P}\Big(D^{*}_{N}(X)\ge (\sqrt{2}\cdot\Sigma_0+1)\frac{A(d,q_0,N)}{N}\Big)\cdot C_0(\Sigma_0,N)\cdot\frac{1}{1-q_0}dq_0
\\&=\int_{q}^{1-e^{\frac{C_1(d,\Sigma_0,N)-1}{C_0(\Sigma_0,N)}}}C_0(\Sigma_0,N)\cdot\frac{1-q_0}{1-q_0}dq_0.
\end{aligned}
\end{equation}

Furthermore, from \eqref{dnsxq1}, we have

$$
    \mathbb{P}\big(D_{N}^{*}\left(Y\right)\ge (\sqrt{2}\cdot\Sigma_0+1)\frac{A(d,q_0,N)}{N}\big)=1-q_1.
$$

Following the steps from \eqref{tsig0} to \eqref{tc1sig0}, we obtain,

$$
    \mathbb{E}[D^{*}_{N}(Y)]=\int_{0}^{1}\mathbb{P}(D^{*}_{N}(Y)\ge t)dt=\int_{q}^{1-e^{\frac{C_1(d,\Sigma_0,N)-1}{C_0(\Sigma_0,N)}}}C_0(\Sigma_0,N)\cdot\frac{1-q_1}{1-q_0}dq_0.
$$

From $q_1< q_0,$ we obtain

$$
    \frac{1-q_1}{1-q_0}>\frac{1-q_0}{1-q_0}.
$$

Hence,

\begin{equation}\label{ewycom1}
    \mathbb{E}(D_{N}^{*}(X))<\mathbb{E}(D_{N}^{*}(Y)).
\end{equation}
\end{proof}

\begin{cor}\label{aep}
For any equivolume partitions $\Omega=\{\Omega_1,\Omega_2,\ldots,\Omega_N\}$ and the corresponding stratified sampling set $P_{\Omega}$. For simple random $d-$dimensional samples $Y=\{Y_{1}, Y_{2}, Y_{3}, \ldots, Y_{N}\}$, we have

\begin{equation}
\mathbb{E}(D_{N}^{*}(P_{\Omega}))< \mathbb{E}(D_{N}^{*}(Y)).
\end{equation}
\end{cor}

\begin{proof}
   For $\Omega=\{\Omega_1, \Omega_2, \ldots, \Omega_N\}$ and $P_{\Omega}$, it follows that 

\begin{equation}
\begin{aligned}
\text{Var}(\sum_{n=1}^{N}\mathbf{1}_{[0,x]}(P_{\Omega}))&=\sum_{i=1}^{N}\frac{|\Omega_i\cap[0,x]|}{|\Omega_i|}(1-\frac{|\Omega_i\cap[0,x]|}{|\Omega_i|})\\&=N|[0,x]|-N^2\sum_{i=1}^{N}(|\Omega_i\cap[0,x]|)^2.
\end{aligned}
\end{equation}

For simple random sampling set $Y=\{Y_1,Y_2,\ldots,Y_N\}$, we have

\begin{equation}
\text{Var}(\sum_{n=1}^{N}\mathbf{1}_{[0,x]}(Y))=N|[0,x]|-N^2|[0,x]|^2.
\end{equation}

According to the proof process of Theorem \ref{uniformnoise0}, we have

\begin{equation}
   \text{Var}(\frac{1}{N}
\sum_{n=1}^{N}\mathbf{1}_{[0,x]}(P_{\Omega}))<\text{Var}(\frac{1}{N}
\sum_{n=1}^{N}\mathbf{1}_{T}(Y_n)).
\end{equation}

and the desired result.
\end{proof}

\begin{rem}
In fact, Corollary \ref{aep} extends the result of Theorem \ref{uniformnoise0} to the more general case of equivolume partitions. The following result is a clear example.
\end{rem}

\textbf{An example:} For a grid-based equivolume partition in two dimensions, the two squares in the upper right corner are merged to form a rectangle $$I=[a_1,a_1+2b]\times [a_2,a_2+b], $$where $a_1,a_2,b$ are three positive constants. 

For the merged rectangle $I$, we use a series of straight-line partitions to divide the rectangle into two equal-volume parts, which will be converted to a one-parameter model if we set the angle between the dividing line and the horizontal line across the center $\theta$, where we suppose that $0\leq\theta\leq\frac{\pi}{2}$. From simple calculations, we can conclude that the arbitrary straight line must pass through \textbf{the center of the rectangle}. For convenience of notation, we set this partition model $\Omega_{\sim}=(\Omega_{1,\sim},\Omega_{2,\sim},Q_3,\ldots,Q_{N})$ in the two-dimensional case, see Figure \ref{fig6}.

\begin{figure}[H]
\centering
\includegraphics[width=0.30\textwidth]{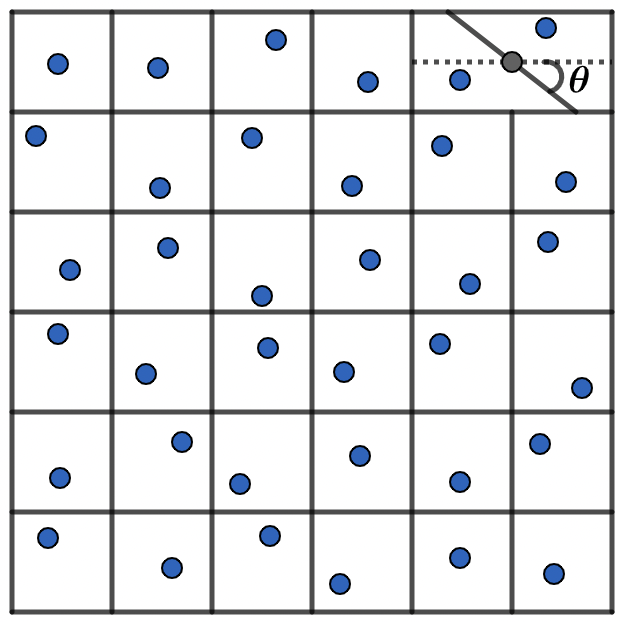}
\caption{\label{fig6}A class of equivolume partitions for two dimensions}
\end{figure}

Now, we consider the $d-$dimensional cuboid
\begin{equation}
    I_d=I\times\prod_{i=3}^{d}[a_i,a_i+b]
\end{equation} and its partitions $\Omega'_{\sim}=(\Omega'_{1,\sim},\Omega'_{2,\sim})$ into two closed, convex bodies with

\begin{equation}
    \Omega'_{1,\sim}= \Omega_{1,\sim}\times \prod_{i=3}^{d}[a_i,a_i+b].
\end{equation}

We choose $a_1=\frac{m-2}{m}, a_2=\frac{m-1}{m}, b=\frac{1}{m}$ in $\Omega'_{1,\sim}$, which is denoted by $\Omega^{*}_{1,\sim}$, and obtain

\begin{equation}\label{omga1}
\Omega^{*}_{\sim}=(\Omega^{*}_{1,\sim},\Omega^{*}_{2,\sim},Q_3 \ldots,Q_{N}).
\end{equation}

Then, we have the following result.

\begin{cor}\label{aneqopart}
For any $0\leq\theta\leq\frac{\pi}{2}$. For the class of equivolume partitions $\Omega^{*}_{\sim}=(\Omega^{*}_{1,\sim},\Omega^{*}_{2,\sim},Q_3 \ldots,Q_{N})$ and the corresponding stratified sampling set $P_{\Omega^{*}_{\sim}}$. For simple random samples $Y=\{Y_{1}, Y_{2}, Y_{3}, \ldots, Y_{N}\}$, we have

\begin{equation}
\mathbb{E}(D_{N}^{*}(P_{\Omega^{*}_{\sim}}))< \mathbb{E}(D_{N}^{*}(Y)).
\end{equation}
\end{cor}

\subsection{Sampling theorem in Hilbert space}

The sampling problem attempts to reconstruct or approximate a function $f$ from its sampled values on the sampling set \cite{BG2}. To solve this problem, it is necessary to specify the function space, since the theoretical and numerical analysis may differ for different function spaces. Common sampling function spaces are: (i)band-limited signals \cite{BG3}; (ii)signals in some shift-invariant spaces \cite{AG1,FX2014,FX2019,GS1,XL1,JBY1}; (iii)non-uniform splines \cite{swc1,swc2}; (iv)sum of some of the above signals \cite{sqy1}, and so on.

Well-posedness of the sampling problem implies that the following inequalities must hold:

\begin{equation}\label{samineq1}
    A\|f\|_{2}^{2}\leq \sum_{j=1}^{N}|f(x_j)|^2\leq  B\|f\|_{2}^{2},
\end{equation}
where $A$ and $B$ are some positive constants. The inequalities \eqref{samineq1} indicate that a small change in a sampled value $f(x_j)$ causes only a small change in $f$. This implies that the sampling is stable, or equivalently, that the reconstruction of $f$ from its samples is continuous. Due to the randomness in the choice of the sampling points, our goal is to choose the following probability estimate:

\begin{equation}
    \mathbb{P}(A\|f\|_{2}^{2}\leq \sum_{j=1}^{N}|f(x_j)|^2\leq  B\|f\|_{2}^{2})\ge 1-\epsilon,
\end{equation}
where $\epsilon>0$ can be taken as arbitrarily small.

In this subsection, we study the problem of stratified sampling sets, and give explicit bounded expressions for a set of sampling inequalities in general Hilbert spaces.

\begin{thm}
For integers $b\ge 1$, $m\ge 1$ and $N=b^m$. Then for well-defined $d-$dimensional Hilbert space filling curve samples $X_i\sim U(E_i), i=1,2,\ldots, N=b^m$, we have

\begin{align*}
(1-\frac{D_{N}^{*}(X)}{\gamma})\|f\|_{L^2}^2\leq \frac{1}{N}\sum_{j=1}^{N}f^{2}(x_j)\leq (\frac{D_{N}^{*}(X)}{\gamma}+1)\|f\|_{L^2}^2,
\end{align*}
holds with probability at least $1-\epsilon$, where $0<\gamma\leq 1$, $f$ belongs to Hilbert space.
\end{thm}

\begin{proof}
Consider the following kernel function,
$$K(x,y)=\int_{[0,1]^{d}}\mathbf{1}_{(x,1]}(t)
\mathbf{1}_{(y,1]}(t)dt.$$

It is not realistic to measure exactly for samples $f(x_j),1\leq j\leq N$, see discussions in \cite{AG1}. A better assumption is that the sampling number has the following form:
\begin{equation}\label{opform1}
\begin{aligned}
f^2(x)=\int_{[0,1]^{d}}K(x,y)g^2(y)dy,
\end{aligned}
\end{equation}
where $f,g$ belongs to Hilbert space. For samples $x_{j},1\leq j\leq N,$ $K(x_j,y)$ is a set of functionals acting on the function $g^2$ to produce the data $f^2(x_j)$. The $K(x_j,y)$ functionals can reflect the characteristics of the sampling devices.

Therefore, we have
\begin{equation}\label{rkhs1}
\begin{aligned}
&|\frac{1}{N}\sum_{j=1}^{N}f^2(x_j)
-\int_{[0,1]^{d}}f^2(x)dx|\\=&|\frac{1}{N}
\sum_{j=1}^{N}\int_{[0,1]^{d}}K(x_j,y)g^2(y)dy-
\int_{[0,1]^{d}}\int_{[0,1]^{d}}K(x,y)g^2(y)dydx|
\\=&|\int_{[0,1]^d}\sum_{j=1}^{N}\frac{1}{N}
K(x_j,y)g^2(y)dy-
\int_{[0,1]^{d}}g^2(y)\int_{[0,1]^{d}}K(x,y)dxdy|,
\end{aligned}
\end{equation}
and 

\begin{equation}
    \begin{aligned}
    &|\int_{[0,1]^d}\sum_{j=1}^{N}\frac{1}{N}
K(x_j,y)g^2(y)dy-
\int_{[0,1]^{d}}g^2(y)\int_{[0,1]^{d}}K(x,y)dxdy|
\\=&|\int_{[0,1]^{d}}
\Big(\sum_{j=1}^{N}\frac{1}{N}K(x_j,y)
-\int_{[0,1]^{d}}K(x,y)dx\Big)g^2(y)dy|
\\ \leq&\|\sum_{j=1}^{N}\frac{1}{N}K(x_j,y)
-\int_{[0,1]^{d}}K(x,y)dx\|_{\infty}\|g\|_{L^2}^2.
    \end{aligned}
\end{equation}

Therefore, according to \eqref{opform1}, we obtain $\gamma\|g\|_{L^2}^2\leq \|f\|_{L^2}^2\leq \|g\|_{L^2}^2,$
where $0<\gamma\leq 1$.

In addition, combining with theorem \ref{convord1}, we have
\begin{equation}\label{rkhs2}
\begin{aligned}
&\|\sum_{j=1}^{N}\frac{1}{N}K(x_j,y)
-\int_{[0,1]^{d}}K(x,y)dx\|_{\infty}\\ \leq& \sup_{y\in[0,1]^{d}}|\int_{[0,1]^{d}}\sum_{j=1}^{N}
\frac{1}{N}\mathbf{1}_{(x_j,1]}(t)
\mathbf{1}_{(y,1]}(t)dt-\int_{[0,1]^{d}}
\int_{[0,1]^{d}}
\mathbf{1}_{(x,1]}(t)\mathbf{1}_{(y,1]}(t)dtdx|\\
=&\sup_{y\in[0,1]^{d}}|\int_{[0,1]^{d}}\Big(\sum_{j=1}^{N}
\frac{1}{N}\mathbf{1}_{(x_j,1]}(t)-\int_{[0,1]^{d}}
\mathbf{1}_{(x,1]}(t)dx\Big)\mathbf{1}_{(y,1]}(t)dt|\\ \leq
&\sup_{y\in[0,1]^{d}}\|\sum_{j=1}^{N}
\frac{1}{N}\mathbf{1}_{(x_j,1]}(t)-\int_{[0,1]^{d}}
\mathbf{1}_{(x,1]}(t)dx\|_{\infty}
\int_{[0,1]^{d}}\mathbf{1}_{\mathbf{1}_{(y,1]}}(t)dt\\ \leq
&\|\sum_{j=1}^{N}
\frac{1}{N}\mathbf{1}_{(x_j,1]}(t)-\int_{[0,1]^{d}}
\mathbf{1}_{(x,1]}(t)dx\|_{\infty}\\ \leq &\sup_{t\in[0,1]^{d}}|\sum_{j=1}^{N}
\frac{1}{N}\mathbf{1}_{(x_j,1]}(t)-\int_{[0,1]^{d}}
\mathbf{1}_{(x,1]}(t)dx|\\ \leq &D_{N}^{*}(x_1,x_2,\ldots,x_N).
\end{aligned}
\end{equation}
holds with probability at least $\epsilon$.

Combining with \eqref{rkhs1}-\eqref{rkhs2}, the desired result is obtained.

\end{proof}

\subsection{Application on Koksma–Hlawka inequality}
The classical K-H inequality is not suitable for functions with simple discontinuities. While \cite {GBCT} proposes a variant K-H type inequality that is suitable for piecewise smooth functions of $f\cdot\mathbf{1}_{\Omega}$, where $f$ is smooth and $\Omega$ is a Borel convex subset of $[0,1]^{d}$. In this subsection, we will use the star discrepancy upper bound of Hilbert space filling curve sampling to give an approximation error in the piecewise smooth function space. First of all, there is the following lemma:

\begin{lma}\cite{GBCT}\label{vkh}
Let $f$ be a piecewise smooth function on $[0,1]^{d}$, and $\Omega$ be a Borel subset of $[0,1]^{d}$. Then for $x_{1},x_{2},\ldots,x_{N}$ in $[0,1]^{d}$, we have

\begin{equation}
|\frac{\sum_{n=1}^{N}(f\cdot\mathbf{1}_{\Omega})(x_{n})}{N}-
\int_{\Omega}f(x)dx|\leq D_{N}^{\Omega}(x_{1},x_{2},\ldots,x_{N})\cdot V(f),
\end{equation}
where

\begin{equation}\label{Borel}
 D_{N}^{\Omega}(x_{1},x_{2},\ldots,x_{N})=2^{d}\sup_{A\in[0,1]^{d}}
 |\frac{\sum_{n=1}^{N}\mathbf{1}_{\Omega\cap A}(x_{n})}{N}-\lambda(\Omega\cap A)|,
\end{equation}
and

\begin{equation}\label{vf0}
V(f)=\sum_{u\subset\{1,2,\ldots,d\}}2^{d-|u|}\int_{[0,1]^{d}}
|\frac{\partial^{|u|}}{\partial x_{u}}f(x)|dx.
\end{equation}
The symbol $\frac{\partial^{|u|}}{\partial x_{n}}f(x)$ is the partial derivative of $f$ with respect to those components of $x$ with index in $u$, and the supremum is taken over all axis-parallel boxes $A$.
\end{lma}

\begin{thm}\label{convexsetapp}
For integers $b\ge 1$, $m\ge 1$ and $N=b^m$. Let $f$ be a piecewise smooth function on $[0,1]^{d}$, and $\Omega$ be a Borel convex subset of $[0,1]^{d}$. Then for well-defined $d-$dimensional Hilbert space filling curve samples $X_i\sim U(E_i), i=1,2,\ldots, N=b^m$, we have 

\begin{equation}
|\frac{\sum_{n=1}^{N}(f\cdot\mathbf{1}_{\Omega})(X_{n})}{N}-
\int_{\Omega}f(x)dx|\leq 2^{d}D_{N}^{*}(X_1,X_2,\ldots,X_N)\cdot V(f),
\end{equation}
where $D_{N}^{*}(X_1,X_2,\ldots,X_N)$ was obtained in Theorem \ref{convord1}.

Where

\begin{equation}\label{vf}
V(f)=\sum_{u\subset\{1,2,\ldots,d\}}2^{d-|u|}\int_{[0,1]^{d}}
|\frac{\partial^{|u|}}{\partial x_{u}}f(x)|dx.
\end{equation}
The symbol $\frac{\partial^{|u|}}{\partial x_{n}}f(x)$ is the partial derivative of $f$ with respect to those components of $x$ with index in $u$.
\end{thm}

\begin{proof}
Because $\Omega\subset [0,1]^{d}$ is a convex set, and in \eqref{Borel}, $A$ is a convex set, so the test set $\Omega\cap A$ is a convex set and $\mathscr{M}(\partial( \Omega\cap A))\leq 2d$. In addition, for the set $A_1$ and $A_1\subset A_2$, and which satisfy $\lambda(A_2)-\lambda(A_1)\leq \frac{1}{N}$, then $\lambda(\Omega\cap A_2)$ and $\lambda(\Omega\cap A_1)$ constitute $\frac{1}{N}-$ cover. Through $\lambda(\Omega\cap A_2)-\lambda(\Omega\cap A_1)\leq \frac{1}{N}$, we can obtain the cover number. Furthermore, according to the idea of the proof of theorem \ref{convord1}, we get the desired result.
\end{proof}

\textbf{An example:}
For $\epsilon>0$, let $\Sigma$ be the simplex

$$\Sigma=\{(x_1,x_2,\ldots,x_d)\in [0,1]^d:x_1\ge\ldots \ge x_d\ge \epsilon, 1-x_1-\ldots-x_d\ge \epsilon\}.$$

Define 

$$f(x_1,x_2,\ldots,x_d)=\frac{1}{x_1x_2\ldots x_d}(1-x_1-x_2-\ldots-x_d).$$

Then one can show that

$$\sum_{|\alpha|\leq d}\int_{\Sigma}
|(\frac{\partial}{\partial x})^{\alpha}f(x)|dx\leq c\dot \epsilon^{-d}.$$

By theorem \ref{convexsetapp}, for Hilbert space filling curve samples $X_i, i=1,2,\ldots,N$, and all convex sets $\Omega$ contained in $\Sigma$, we have

\begin{equation}
|\frac{\sum_{n=1}^{N}(f\cdot\mathbf{1}_{\Omega})(X_{n})}{N}-
\int_{\Omega}f(x)dx|\leq c\dot \epsilon^{-d}\cdot2^{d}\cdot D_{N}^{*}(X_1,X_2,\ldots,X_N).
\end{equation}

\section{Conclusions}\label{discuss}

The convergence of the star discrepancy bound for HSFC-based sampling is $O(N^{-\frac{1}{2}-\frac{1}{2d}}\cdot (\ln N)^{\frac{1}{2}})$, which matches the rate using jittered sampling sets and improves the rate using the classical MC method. The strict condition for the sampling number $N=m^d$ in the jittered case is removed, and the applicability of the new stratified sampling method for higher dimensions is improved. Although a faster convergence of the upper bound leads to a faster integration approximation rate, the strict comparison of the size of random star discrepancy under different stratified sampling models remains unresolved.


\begin{thebibliography}{ }

\bibitem{hlawka1961funktionen} E. Hlawka, Funktionen von beschr{\"a}nkter variatiou in der theorie der gleichverteilung, \emph{Annali di Matematica Pura ed Applicata}, 54(1961), 325-333.

\bibitem{koksma1942een} JF. Koksma, Een algemeene stelling uit de theorie der gelijkmatige verdeeling modulo 1, \emph{Mathematica B (Zutphen)}, 11(1942/1943), 7-11.

\bibitem{2010Digital} J. Dick and F. Pillichshammer, Digital nets and sequences: Discrepancy theory and quasi-Monte Carlo integration, \emph{Cambridge University Press}, 2010.

\bibitem{Blue-Noise}A. G. M. Ahmed, H. Perrier and D. Coeurjolly, et al, Low-discrepancy blue noise sampling, \emph{ACM Trans. Graph.}, 35(6)(2016), 1-13.

\bibitem{Cristiano}C. Cervellera and M. Muselli, Deterministic design for neural network learning: An approach based on discrepancy, \emph{IEEE Trans. Neural Netw.}, 15(3)(2004), 533-544.

\bibitem{phdpaper}Y. Lai, Monte Carlo and Quasi-Monte carlo methods and their applications, Ph.D. Dissertation, Department of Mathematics, Claremont Graduate University, California, USA, 1998.

\bibitem{rankla}Y. Lai, Intermediate rank lattice rules and applications to finance, \emph{Appl. Numer. Math.}, 59(2009), 1-20.

\bibitem{Heinrich2001}S. Heinrich, E. Novak, G. W. Wasilkowski and H. Wo\'zniakowski, The inverse of the star-discrepancy depends linearly on the dimension, \emph{Acta Arith.}, 96(3)(2001), 279-302.

\bibitem{aistleitner2011covering}C. Aistleitner, Covering numbers, dyadic chaining and discrepancy, \emph{J. Complexity}, 27(6)(2011), 531-540.

\bibitem{Gnewuch2020}M. Gnewuch, N. Hebbinghaus, Discrepancy bounds for a class of negatively dependent random points including Latin hypercube samples, \emph{Ann. Appl. Probab.}, 31(4)(2021), 1944-1965.

\bibitem{aistleitner2014}C. Aistleitner and M. Hofer, Probabilistic discrepancy bound for Monte Carlo point sets, \emph{Math. Comp.}, 83(2014), 1373-1381.

\bibitem{jittsamp}F. Pausinger and S. Steinerberger, On the discrepancy of jittered sampling, \emph{J. Complexity}, 33(2016), 199-216.

\bibitem{Doerr2}B. Doerr, A sharp discrepancy bound for jittered sampling, \emph{Math. Comp.}, 91(2022), 1871-1892.

\bibitem{HO2016}Z. He, A. B. Owen, Extensible grids: uniform sampling on a space filling curve, \emph{J. R. Stat. Soc. Ser. B}, 78(2016), 917-931.

\bibitem{HZ2019}Z. He, L. Zhu, Asymptotic normality of extensible grid sampling, \emph{Stat. Comput.}, 29(2019), 53-65.

\bibitem{ACV}L. Ambrosio, A. Colesanti, and E. Villa, Outer Minkowski content for some classes of closed sets, \emph{Math. Ann.}, 342(4)(2008), 727–748.

\bibitem{Doerr}B. Doerr, M. Gnewuch, A. Srivastav, Bounds and constructions for the star-discrepancy via $\delta-$covers, \emph{J. Complexity}, 21(2005), 691-709.

\bibitem{Gnewuch2008Bracketing}M. Gnewuch, Bracketing numbers for axis-parallel boxes and applications to geometric discrepancy, \emph{J. Complexity}, 24(2)(2008), 154-172.

\bibitem{FDX2007}F. Cucker and D. X. Zhou, Learning theory: an approximation theory viewpoint, \emph{Cambridge University Press.},  2007.

\bibitem{KP}M. Kiderlen and F. Pausinger, On a partition with a lower expected $L_2$-discrepancy than classical jittered sampling, \emph{J. Complexity} 70(2022), https://doi.org/10.1016/j.jco.2021.101616.

\bibitem{KP2}M. Kiderlen and F. Pausinger, Discrepancy of stratified samples from partitions of the unit cube, \emph{Monatsh. Math.}, 195(2021), 267-306.

 \bibitem{Kuo2005}F. Y. Kuo and I. H. Sloan, Lifting the curse of dimensionality, \emph{Notices of the American Mathematical Society}, 52(11)(2005), 1320–1328.

 \bibitem{Ems}E. Novak and H. Wo\'zniakowski, Tractability of Multivariate Problems, Volume II: Standard Information for Functionals, European Mathematical Society, 2010.

 \bibitem{BG2}R. F. Bass and K. Gr{\"o}chenig, Random sampling of bandlimited functions, \emph{Israel J. Math.}, 177(2010), 1-28.

 \bibitem{BG3}R. F. Bass and K. Gr{\"o}chenig, Relevant sampling of band-limited functions, \emph{Illinois J. Math.}, 57(2013), 43-58.

 \bibitem{GS1}K. Gr{\"o}chenig and H. Schwab, Fast local reconstruction methods for nonuniform sampling in shift-invariant spaces, \emph{SIAM J. Matrix Anal. Appl.}, 24(2003), 899-913.

 \bibitem{sqy1}Q. Y. Sun, Frames in spaces with finite rate of innovation, \emph{Adv. Comput. Math.}, 28(2008), 301-329.

\bibitem{swc1}W. C. Sun, Local sampling theorems for spaces generated by splines with arbitrary knots, \emph{Math. Comput.}, 78(2009), 225-239.

\bibitem{swc2}W. C. Sun and X. W. Zhou, Characterization of local sampling sequences for spline subspaces, \emph{Adv. Comput. Math.}, 30(2009), 153-175.

\bibitem{XL1}J. Xian and S. Li, Sampling set conditions in weighted finitely generated shift-invariant spaces and their applications, \emph{Appl. Comput. Harmon. Anal.}, 23(2007), 171-180.

\bibitem{JBY1}J. B. Yang, Random sampling and reconstruction in multiply generated shift-invariant spaces, \emph{Anal. Appl.}, 17(2019), 323-347.

\bibitem{FX2014}H. F$\ddot{u}$hr and J. Xian, Quantifying invariance properties of shift-invariant spaces, \emph{Appl. Comput. Harmon. Anal.}, 36(2014), 514-521.

\bibitem{FX2019}H. F$\ddot{u}$hr and J. Xian, Relevant sampling in finitely generated shift-invariant spaces, \emph{J. Approx. Theory}, 240(2019), 1-15.

\bibitem{AG1}A. Aldroubi and K. Gr{\"o}chenig, Nonuniform sampling and reconstruction in shift-invariant spaces, \emph{SIAM Rev.}, 43(2001), 585-620.

\bibitem{GBCT}G. Gigante, L. Brandolini, L. Colzani and G. Travaglini, On the koksma–hlawka
inequality, \emph{J. Complexity}, 29(2)(2013), 158–172.

\end{thebibliography}
\end{document}